\documentclass{article}
\usepackage{amsthm,amsmath,amssymb,amsfonts,latexsym,enumerate,vmargin}
\usepackage[usenames,dvipsnames]{color}
\usepackage{hyperref}
\usepackage{tikz}
\usetikzlibrary{arrows}

\setmarginsrb{30mm}{20mm}{30mm}{20mm}{0pt}{10mm}{0pt}{10mm}

\def\N{\mathbb{N}}

\def\Z{\mathbb{Z}}

\newtheorem{thm}{Theorem}
\newtheorem*{thm*}{Theorem}

\newtheorem*{claim*}{Claim}

\newtheorem*{dfn*}{Definition}
\newtheorem{lemma}[thm]{Lemma}
\newtheorem*{lemma*}{Lemma}
\newtheorem{prop}[thm]{Proposition}
\newtheorem*{prop*}{Proposition}
\newtheorem{cor}[thm]{Corollary}
\newtheorem*{cor*}{Corollary}

\newtheorem*{conj*}{Conjecture}
\theoremstyle{remark}

\newtheorem*{rmk*}{Remark}

\newtheorem*{rmks*}{Remarks}

\newcommand{\notlinkgraph}{
{\tikzstyle{every node}=[circle, draw, fill=black!20, inner sep=0pt, minimum width=3pt]
\begin{tikzpicture}[thick,baseline=-0pt]
\node at(0.0,0.0)(0){}; \node at(0.2,0.1)(1){}; \node at(0.0,0.2)(2){}; \node at(0.4,0.2)(3){}; 
\node at(0.4,0.0)(4){}; \node at(0.6,0.1)(5){}; \node at(0.8,0.2)(6){}; \node at(0.8,0.0)(7){}; 
\draw { (0)--(1)--(4)--(5)--(7)  (1)--(2)  (3)--(4)  (5)--(6)};
\end{tikzpicture}}}

\title{Triply Existentially Complete \\ 
Triangle-Free Graphs}
\author{
Chaim Even-Zohar\thanks{
Department of Mathematics, Hebrew University, Jerusalem 91904, Israel. \newline
e-mail: \href{mailto:chaim.evenzohar@mail.huji.ac.il}{chaim.evenzohar@mail.huji.ac.il}~.}
\and  
Nati Linial\thanks{
Department of Computer Science, Hebrew University, Jerusalem 91904, Israel.\newline
e-mail: \href{mailto:nati@cs.huji.ac.il}{nati@cs.huji.ac.il}~.
Supported by ISF and BSF grants.}
}

\begin{document}

\maketitle

\begin{abstract}
A triangle-free graph $G$ is called \emph{$k$-existentially complete} if 
for every induced $k$-vertex subgraph $H$ of $G$, 
every extension of $H$ to a $(k+1)$-vertex triangle-free graph can be realized 
by adding another vertex of $G$ to $H$.
Cherlin~\cite{cherlin1993,cherlin2011} asked whether
$k$-existentially complete triangle-free graphs exist for every $k$.
Here, we present known and new constructions of $3$-existentially complete triangle-free graphs.
\end{abstract}

\section{Introduction}\label{introsect}

It is well known that the Rado graph $R$~\cite{erdos1963,rado1964} 
is characterized by being existentially complete. 
Namely, for every two finite disjoint subsets of vertices $A, B\subset V(R)$ 
there is an additional vertex $x$ that is adjacent to all vertices in $A$ and to none in $B$. 
Many variations on this theme suggest themselves, 
specifically concerning finite and $H$-free graphs. 
First, the aforementioned extension property of $R$ suggests 
a search of small finite graphs that satisfy this condition whenever $|A|+|B| \le k$.
It is known that Paley graphs and finite random graphs of order $\exp(O(k))$ 
have this property~\cite{blass1981,bollobas1981,cameron2002,erdos1963,bonato2009}.  
%????  Is pseudo randomness enough for this purpose   
%????  -- No, Cameron and Stark (2002)

The Rado graph is also {\em homogeneous}, that is to say, every 
isomorphism between two of its finite subgraphs can be extended to an automorphism of $R$. 
Henson~\cite{henson1971} has discovered 
the generic infinitely countable {\em triangle-free} graph $R_3$ that is homogeneous as well. 
It is uniquely defined by the following property.
Given two finite disjoint subsets $A, B\subset V(R_3)$ with $A$ {\em independent},
there is an additional vertex $x$ that is adjacent to all vertices in $A$ and to none in $B$.

The graph $R_3$ suggests an extremely interesting question that was
raised and studied by Cherlin~\cite{cherlin1993,cherlin2011}.
Namely, do there exist finite graphs with a similar property? 
We say that a triangle-free graph is $k$-{\em existentially complete}
if it satisfies this condition whenever $|A|+|B| \le k$.
Are there $k$-existentially complete triangle-free graphs for every $k$?
A weaker variant of this question had been raised by Erd\H{o}s and Pach~\cite{pach1981,erdos1985}. 
Similar questions for graphs and other combinatorial structures appear in \cite{alspach1991, bonato2000, bonato2001}. 
In the literature one occasionally finds the shorthand \emph{$k$-e.c.} as well as the alternative term \emph{$k$-existentially closed}.

Cherlin's question can be viewed as an instance of a much wider subject in graph theory, 
namely understanding the extent to which the local behavior of infinite graphs 
can be emulated by finite ones.
Here are some other instances of this general problem.

The infinite $d$-regular tree $\mathbb{T}_d$ is the ultimate $d$-regular expander 
in at least two senses. 
It has the largest possible number of edges emanating from every finite set of vertices. 
It also has the largest spectral gap that a $d$-regular graph can have. 
This observation suggests the search for finite arbitrarily large Ramanujan graphs, 
and for the limits on expansion in finite $d$-regular graphs. 
$\mathbb{T}_d$ is also, of course, acyclic that leads to the question 
how large the girth can be in a finite $d$-regular graph. 

There are several examples of local conditions that 
can be satisfied in an infinite graph but not in a finite one. 
A nice example comes from an article of Blass, Harari and Miller~\cite{blass1980}. 
They define the link of a vertex to be the subgraph induced by its neighbors, 
and consider graphs in which all links are isomorphic to some fixed graph $H$. 
As they observe, there is an infinite graph with all links isomorphic to $H=$\notlinkgraph~, 
but this is impossible for finite graphs. 

For all the currently known $k$-existentially complete triangle-free graphs, 
the parameter $k$ is bounded by $3$. 
The question is wide open for $k \geq 4$.
The known finite graphs with the $R$-like extension property 
do not seem adjustable to the triangle-free case. 
In fact we tend to believe that there is some absolute constant $k_0$ 
such that no triangle-free graph is $k_0$-existentially complete.

For general $p \geq 3$, Henson's universal countable $K_p$-free graph $R_p$ 
is defined by the analogous extension property,
where the subset $A$ is only required to be $K_{(p-1)}$-free. 
There is a simple connection between 
the existential completeness properties of $K_p$-free graphs for different values of $p$. 
Thus, if, as we suspect, no finite $k_0$-existentially complete triangle-free graphs exist, 
then no finite $K_4$-free graph can be $(k_0+1)$-existentially complete, etc. 
This follows since the link of every vertex in 
a $(k+1)$-existentially complete $K_p$-free graph is 
a $k$-existentially complete $K_{(p-1)}$-free graph. 
Finite random $(p-1)$-partite graphs provide an easy lower bound,
since they are $K_p$-free and $(p-2)$-existentially complete.

Cherlin gives several examples of 
finite $3$-existentially complete triangle-free graphs, or shortly 3ECTF,
and asks which additional properties they can have.
In particular, for what $\mu_2$ it is possible that 
every $2$ nonadjacent vertices have at least $\mu_2$ common neighbors?
Similarly, can every independent set of $3$ vertices have $\mu_3 > 1$ common neighbors?
Maybe the search for more robust examples for finite 3ECTF graphs,
would shed light on the question of finite 4ECTF graphs and beyond.

The main purpose of this article is to investigate Cherlin's examples and construct some more 3ECTF graphs.
We also study and extend a construction by Erd\H{o}s~\cite{erdos1966} and Pach~\cite{pach1981}.
For $n$ even we show that there are at least $2^{n^2/(16+o(1))}$ 3ECTF graphs on $n$ vertices. 
We show (Corollary~\ref{many1}) that in these graphs the average degree can be as small as $O(\log n)$ and as high as $n/4$.
For all $\mu_2$, we find $2^{\Omega(n^2)}$ distinct 3ECTF graphs of order $\le n$ in
which every two nonadjacent vertices have at least $\mu_2$ common neighbors (Corollary~\ref{many2}).
However, some independent triplets in these graphs may have only one common neighbor.

The constructions are presented in three main phases.
In each of Sections~\ref{albertsect}, \ref{hypersect} and \ref{twistsect} we describe a basic construction, which we then extend in various ways.
Table~\ref{3ECTFtable} exhibits three main parameters of the graph families under discussion:
the number of vertices, the vertex degrees, and $\mu_2$, the minimum number of common neighbors of nonadjacent pairs.

\begin{table}[h]
\centering\renewcommand{\arraystretch}{2}\begin{tabular}{l l c c c}
 & \textbf{Graph} & \textbf{Vertices} & \textbf{Degrees} & \textbf{$\mu_2$} \\
\hline
\hline
\textbf{Albert} 
& $A(n)$, $n\geq 4$~~(\cite{cherlin2011}, Section 13.1) & $4n$ & $n+1$ & 2 \\
& $A_M$, $M \in \{0,1\}^{m \times n}$ * & $2m+2n$ & $m+1,n+1$ & 2 \\
\hline
\textbf{Hypercube} 
& $C_{3k+1}$~~(Erd\H{o}s~\cite{erdos1966}, Pach~\cite{pach1981}) & $2^{3k+1}$ & $\binom{3k}{k}$ ** & $\binom{2k}{k}$ \\
& $C_{3k-1}(m)$, $k \geq 1$, $m \geq 4$ & $m2^{3k-1}$ & $m\binom{3k}{k}$ ** & $\binom{2k}{k}$ \\
& $C_{k,j}$, $1 \leq j \leq k$ & $2^{3k+j}$ & $\binom{3k+j}{k+j}$ ** & $2\binom{2k-1}{k-j}$ \\
\hline
\textbf{Twisted} 
& $G(m_0,...,m_3)$, $m_i \geq 2$ & $4\sum_im_i$ & $\sum_im_i+1$ & $2$ \\
& $G_T(m,k)$, $T$ tournament* & $|T|m2^{3k-1}$ & $|T|m\binom{3k}{k}$ ** & $\binom{2k}{k}$ \\
\hline
\hline
\multicolumn{5}{l}{\footnotesize{* For asymptotically almost every such $M$ or $T$.~~~** Upto a multipilcative constant, for $k \rightarrow \infty$.}} \\
\end{tabular}
\caption{Constructions of 3ECTF graphs}
\label{3ECTFtable}
\end{table}

Although this work is meant to be self-contained, 
the reader is encouraged to consult part III of Cherlin's article~\cite{cherlin2011} 
for more background of the subject and many additional details.
To simplify matters we maintain a graph-theoretic terminology, 
and refrain from using Cherlin's view of maximal triangle-free graphs 
as combinatorial geometries.
We believe that this presentation makes
the structure and the symmetries of the graphs more transparent.

Notation: We denote the fact that vertices $x, y$ are adjacent in the graph under discussion 
by $x \sim y$.
The \emph{neighborhood} of $x$ is $N(x) = \{y \in V | x \sim y\}$.

\section{Existentially Complete Graphs}\label{ecgsect}

Here are some definitions and useful reductions that are due to Cherlin.

\begin{dfn*}[Extension Properties, Section 11.1 of~\cite{cherlin2011}]
Let $k$ be a positive integer.
The following properties of a triangle-free graph $G=(V,E)$ are defined thus: 
\begin{enumerate}
\item \emph{$(E_k)$, [also known as $k$-existentially complete]:}
For every $B \subseteq A \subseteq V$ where $|A|\le k$ and $B$
is independent, there exists a vertex $v$ 
that is adjacent to each vertex of $B$ and to no vertex of $A \setminus B$.
\item \emph{$(E_k')$:}
For every $B \subseteq A \subseteq V$ where $A$ is independent of cardinality {\em exactly} $k$
there exists a vertex $v$, adjacent to each vertex of $B$ and to no vertex of $A \setminus B$.
\\ Also, every independent set with fewer than $k$ vertices 
is contained in an independent set of cardinality $k$. 
\item \emph{$(Adj_k)$:}
Every independent set of cardinality $\le k$ has a common neighbor.
\end{enumerate}
\end{dfn*}

An $(Adj_2)$ graph is also called \emph{maximal triangle-free},
since the addition of any edge would create a triangle.
In other words, the graph has diameter $\le 2$.
In a \emph{twin-free} graph no two vertices have the same neighborhood.

We have the following implications for triangle-free graphs 
(Cherlin~\cite{cherlin2011}, Lemmas 11.2-11.4).

\begin{lemma}\label{ext1}
For $k \geq 2$, the properties $(E_k)$ and $(E_k')$ are equivalent.
\end{lemma}

\begin{lemma}\label{ext2}
For $k \geq 3$, property $(E_k)$ is equivalent to the conjunction of $(Adj_k)$ and $(E_3)$.
\end{lemma}

\begin{lemma}\label{ext3}
A graph has property $(E_2)$ iff it is maximal triangle-free, twin-free, and contains an anti-triangle.
\end{lemma}

In order to bridge between Lemmas~\ref{ext2} and~\ref{ext3},
Cherlin investigates triangle-free graphs that are $(Adj_3)$ and $(E_2)$ but not $(E_3)$.
In Sections 11.4-11.5 of~\cite{cherlin2011} 
these exceptions are described in terms of linear combinatorial geometries,
which can be interpreted as graphs of a certain circular structure.
In general, \emph{circular} graphs can be defined as 
a set of arcs in a cyclically ordered set,
where adjacency means disjointness of the corresponding arcs.
Such graphs, defined by disjointedness in a family of subsets, 
are sometimes called (general) Kneser graphs (e.g.~\cite{matousek2003}).

Here and below we only consider finite graphs.
The circular graph $O_{3n-1}$ is
formed by all arcs of $n$ consecutive elements in $\Z_{3n-1}$.
It was independently introduced several times, e.g.,
Erd\H{o}s and Andr\'asfai~\cite{andrasfai1962}, 
Woodall~\cite{woodall1973},
Pach~\cite{pach1981},
and van den Heuvel~\cite{heuvel1993}.
In terms of rational complete graphs~\cite{hell2004}, 
$O_{3n-1}$ is equivalent to $K_{(3n-1)/n}$, being ''almost'' a triangle.
As shown in~\cite{pach1981,brouwer1995}, 
these are the only finite triangle-free twin-free graphs 
where every independent set has a common neighbor.

By Lemma 11.15 of~\cite{cherlin2011}, the circular graphs $\{O_{3n-1}\}_{n \geq 3}$ are 
the only finite triangle-free graphs that are $(Adj_3)$ and $(E_2)$ but not $(E_3)$.
Moreover, the $5$-cycle $O_5$ and the edge $O_2$ are 
the only maximal triangle-free graphs that are twin-free and do not contain an anti-triangle.
Taking into account these exceptions, Lemmas~\ref{ext2} and~\ref{ext3} easily imply the following corollary.

\begin{cor}[Cherlin, Sections 11.1--11.5 of~\cite{cherlin2011}]\label{eck}
For $k \geq 3$,
a finite triangle-free graph $G$ is $k$-existentially complete
if and only if the following conditions hold.
\begin{enumerate}
\item Every independent set of cardinality $\le k$ has a common neighbor.
\item There do not exist two vertices $x, y$ with $N(x) = N(y)$.
\item $G$ is not isomorphic to $O_{3n-1}$ for $n \geq 1$.
\end{enumerate}
\end{cor}

In Section 12 of~\cite{cherlin2011}, the notion of $(Adj_k)$ is refined to
the $k$-th \emph{multiplicity} of $G$. 
This is done by considering 
the smallest possible number of common neighbors of an independent set of $k$ vertices:
$$ \mu_k(G) \;=\; \min\limits_{\substack{A \subseteq V(G)\text{ independent} \\ |A|=k}}
\# \{b \in V(G) \;|\; N(b) \supseteq A\} $$
Cherlin proves the following chain of implications for 3ECTF graphs:
$$ 
\mu_4(G) \geq 1 \;\Rightarrow\;
\mu_3(G) \geq 5 \;\Rightarrow\;
\mu_3(G) \geq 2 \;\Rightarrow\;
\mu_2(G) \geq 5 \;\Rightarrow\;
\mu_2(G) \geq 2
$$

In these terms, the strongest known example is 
the strongly regular Higman--Sims graph, on $100$ vertices,
for which $\mu_3(HS)=2$ and $\mu_2(HS)=6$.
However, by a beautiful spectral calculation (\cite{cherlin2011}, Section 12.3)
no strongly regular graph has property 4ECTF.
Perhaps the Higman--Sims construction should be viewed as a sporadic example. 
On the other hand, here we introduce a large collection of 3ECTF graphs 
with arbitrarily large $\mu_2(G)$ and $\mu_3(G) = 1$. 
Thus the two lowest levels in this hierarchy are not very restrictive.

\section{Albert Graphs}\label{albertsect}

We turn to describe an infinite sequence of 3ECTF graphs 
which Cherlin attributes to Michael Albert.
The Clebsch graph is a triangle-free strongly regular graph on 16 vertices.
It can be represented as the union of four $4$-cycles,
where each vertex is adjacent as well to its antipodes in the other cycles.
One can check directly that this graph is $3$-existentially complete,
e.g. by Corollary~\ref{eck}.
Albert's construction is the extension of the
Clebsch graph to any number of $4$-cycles.
Formally, Albert's 3ECTF graphs sequence $A(n)$ is defined by
\begin{align*}
&V(A(n)) = \{(i,x) \;|\; i \in \{1,2,...,n\}, x \in \Z_4\} \\
&(i,x) \sim (i, x + 1) \;\; \text{ for all } i,x \\
&(i,x) \sim (i', x + 2) \;\; \text{ for all } x,i \neq i'\;,
\end{align*}
where addition is in $\Z_4$.

This construction was thoroughly generalized by Cherlin, to Albert geometries. 
Here, we offer a different viewpoint of these graphs. Let $m,n \geq 4$ be integers.
An $m \times n$ zero-one matrix $M$ is said to be \emph{shattered} 
if the submatrix corresponding to any three rows or three columns 
contains all four possible patterns $aaa, aab, aba, baa$.
Namely, it must contain 
at least one of the strings $000$ and $111$, and one of $001$ and $110$, and so on.
The \emph{Albert graph} $A_M$ of a shattered matrix $M$ is obtained from an $m$-matching
and an $n$-matching.
The corresponding entries of $M$ tell us how to connect these $2m+2n$ vertices.
\begin{align*}
&V(A_M) = \{a_1,...,a_m\} \cup \{b_1,...,b_m\} \cup \{c_1,...,c_n\} \cup \{d_1,...,d_n\} \\
&a_i \sim b_i \;\; \text{ for all } i \\
&c_j \sim d_j \;\; \text{ for all } j \\
&a_i \sim c_j,\; b_i \sim d_j \;\; \text{ if } M_{ij}=1 \\
&a_i \sim d_j,\; b_i \sim c_j \;\; \text{ if } M_{ij}=0 \;.
\end{align*}

For example, when $M$ is the $4 \times 4$ identity matrix, $A_M$ is the Clebsch graph.
Albert's construction corresponds to larger identity matrices.

\begin{prop}
If the matrix $M$ is shattered, then the Albert graph $A_M$ is 3ECTF.
\end{prop}

\begin{proof}
We first observe that $A_M$ is triangle-free.
Of any three vertices at least two must either come from 
$U = \{a_1,...,a_m\} \cup \{b_1,...,b_m\}$, 
or from $W = \{c_1,...,c_n\} \cup \{d_1,...,d_n\}$.
But an edge in $U$ must be of the form $(a_i,b_i)$,
and $a_i$ and $b_i$ have no common neighbors.

We proceed to verify the conditions in Corollary~\ref{eck}.
It is easy to see that $A_M$ is twin-free and not an $O_{3n-1}$.
Property $(Adj_3)$ follows since $M$ is shattered.
An independent triplet in $U$ must consist of either $a_i$ or $b_i$, 
either $a_j$ or $b_j$ and either $a_k$ or $b_k$ for some distinct $i,j,k$.
A common neighbor exists thanks to the appropriate pattern in the 
restriction of $M$ to the rows $i$, $j$, and $k$.
For an independent set with two vertices in $U$ and one in $W$,
the neighbor of the latter inside $W$ is adjacent also to the first two.
\end{proof}

The constructions of Albert Geometries in Examples 13.1 of~\cite{cherlin2011},
come from explicit shattered matrices.
Random matrices can be used as well.
Thus, simple counting arguments yield

\begin{cor}\label{many1}$ $
\begin{enumerate}
\item
Almost every zero-one $m \times n$ matrix with $m \ge C\log n$ and $n \ge C\log m$ 
is shattered. Here $C>0$ is some absolute constant.
\item
For even $n$, the number of $n$-vertex 3ECTF graphs is at least $2^{n^2/(16+o(1))}$.
\end{enumerate}
\end{cor}

This gives some insight on the possible behavior of 3ECTF graphs.
On the one hand, taking $m=n$ we get a regular graph of degree $|V|/4+1$.
On the other hand, if $m = \Theta(\log n)$ the graph
has $\Theta(\log|V|)$ vertices of degree close to $|V|/2$ 
while most vertices have only $\Theta(\log|V|)$ neighbors.
As for multiplicities, in every Albert graph $\mu_2(G) =  2$ and $\mu_3(G) =  1$.

\section{Hypercube Graphs}\label{hypersect}

Recall that the Clebsch graph is a Cayley graph of $\Z_2^4$, with
generators the unit vectors and the all-ones vector.
Equivalently, $x \sim y$ iff their Hamming distance $d(x,y)$ is $1$ or $4$.
Following Franek and R\"odl~\cite{franek1993},
we denote this graph by $\left\langle \Z_2^4, \{1,4\} \right\rangle$.
Note that it also equals $\left\langle \Z_2^4, \{3,4\} \right\rangle$.
Here, we consider the graphs
$$ C_{3k+1} = \left\langle \Z_2^{3k+1}, \{2k+1,2k+2,...,3k+1\} \right\rangle $$
which Erd\H{o}s~\cite{erdos1966} used in the study of Ramsey numbers.
As mentioned, $C_4$ is the Clebsch graph. 
The extension properties of these graphs were studied by Pach~\cite{pach1981}.
For future use, we record a variant of his argument in the following lemma.

\begin{lemma}\label{tri}
If  $x,y,z  \in \Z_2^n$ satisfy
$$
d(x,y) \leq a + b ,\;\;\;\;\;
d(x,z) \leq a + c ,\;\;\;\;\;
d(y,z) \leq b + c .
$$
for some integers $a, b, c \ge 0$,
then there is some $v \in \Z_2^n$ for which
$$
d(v,x) \leq a ,\;\;\;\;\;
d(v,y) \leq b ,\;\;\;\;\;
d(v,z) \leq c .
$$
\end{lemma}

\begin{proof}
Define the vector $m$ by the coordinate-wise majority vote of $x$, $y$, and $z$. 
Note that the three vectors $x'=x+m$, $y'=y+m$, and $z'=z+m$ have disjoint supports. 
We find a $v'$ satisfying the claim for these three vectors and let $v=v'+m$.

If the Hamming weights satisfy $w(x') \le a$, $w(y')\le b$ and $w(z')\le c$, then
we are done by taking $v'=0$. 
Otherwise, by assumption, at most one of these inequalities can be violated,
say $w(x') > a$. We take $v'$ to have weight $w(v') = w(x') - a$ and satisfy $x' \ge v'$ 
coordinate-wise.
Obviously $d(v',x') \leq a$. But also,
$$ d(v',y') \leq w(v') + w(y') = w(x') - a + w(y') = d(x',y') - a \leq b ,$$
and similarly $d(v',z') \leq c$.
\end{proof}

\begin{prop}\label{hc1}
The graphs $C_{3k+1}$ are 3ECTF.
\end{prop}

\begin{proof}
To see that $C_{3k+1}$ is triangle-free, 
suppose $x \sim y \sim z$, then $ d(x,z) \leq d(x,\bar{y}) + d(\bar{y},z) \leq k + k = 2k$, 
where $\bar{y}$ is $y$'s antipode, namely $\bar{y} = y + (1,...,1)$.
Therefore $x \not \sim z$.

Now we check the conditions of Corollary~\ref{eck}. 
Clearly $N(x)$ is not equal to $N(y)$ for $x \neq y$, 
since both are distinct Hamming balls of the same radius.
Property $(Adj_3)$ follows using Lemma~\ref{tri}.
Indeed, if the set $\{x,y,z\}$ is independent, apply the lemma with $a=b=c=k$.
For $v$ as in the lemma, $\bar{v}$ is at least $2k+1$ away from each of the three vectors.
\end{proof}

\begin{prop}\label{hc2}
$\mu_2(C_{3k+1}) = \binom{2k}{k}$, $\mu_3(C_{3k+1}) = 1$.
\end{prop}

\begin{proof}
We check that every two nonadjacent vertices in $C_{3k+1}$ 
have at least $\binom{2k}{k}$ common neighbors.
Let $x$ and $y$ be two vectors of even Hamming distance $2t$ where $1 \leq t \leq k$.
We count some of the common neighbors of $x$ and $y$.
Of the $3k+1-2t$ coordinates in which they agree,
we flip $2k+1-t$ coordinates of our choice, and keep the other $k-t$ unchanged.
The remaining $2t$ coordinates are divided equally, $t$ as in $x$ and $t$ as in $y$.
Each resulting vector is at Hamming distance $2k+1$ from both $x$ and $y$, 
which hence have at least
$$ \binom{3k+1-2t}{k-t}\binom{2t}{t} $$ 
distinct joint neighbors.
This expression is decreasing in $t$, 
and hence always $\ge\binom{2k}{k}$, with equality for $t=k$.

For $d(x,y) = 2t-1$ odd, we argue similarly.
We flip $2k+1-(t-1)$ of the $3k+1-(2t-1)$ common coordinates, 
and divide the other $2t-1$ to $t$ and $t-1$.
This yields
$$ 2\binom{3k+1-(2t-1)}{k-t}\binom{2t-1}{t} $$
common neighbors. The minimum, $\binom{2k}{k}$, is again attained at $t=k$.

Finally, it is easy to demonstrate three vectors in $C_{3k+1}$ with a single joint neighbor.
Take three vectors, $x$, $y$, and $z$, of Hamming weight $k$ and disjoint supports.
For a common neighbor $v$ we have
$$ 3(2k+1) \leq d(x,v) + d(y,v) + d(z,v) \leq 2 \cdot 3k + 3 = 3(2k+1) $$
since in $3k$ coordinates at most two of the distances can contribute.
But equality is reached only by the all-ones vector $v = (1,...,1)$.
\end{proof}

By Propositions~\ref{hc1}--\ref{hc2},
the sequence $C_{3k+1}$ constitutes an example for 3ECTF graphs 
with $\mu_2(C_{3k+1}) \rightarrow \infty$. Here $\mu_2=n^{2/3-o(1)}$, where $n = 2^{3k+1}$ is the number of vertices.
Also, these graphs are $n^{\lambda-o(1)}$-regular, 
where $\lambda = \log_2 3 - \frac 23 \approx 0.918$.
A neighborhood of a vertex in these graphs is also a largest possible independent set.
See~\cite{erdos1966}.

Here are two variations where the 2-multiplicity and the degree are traded off.
We first apply Albert's idea to $C_{3k+1}$.
For $x \in \Z_2^k$ denote $parity(x) = \sum_{i=1}^{k} x_i \text{ mod } 2$.
We first partition $C_{3k+1}$ as follows.
\begin{align*}
&V_1 = \{x \in \Z_2^{3k+1}\;|\;parity(x)=x_1=x_2\} \\
&V_2 = \{x \in \Z_2^{3k+1}\;|\;parity(x)\neq x_1=x_2\} \\
&V_3 = \{x \in \Z_2^{3k+1}\;|\;parity(x)=x_1\neq x_2\} \\
&V_4 = \{x \in \Z_2^{3k+1}\;|\;parity(x)=x_2\neq x_1\}
\end{align*}
Of course, we may forget the first two coordinates, 
and regard the elements of each $V_i$ as $\Z_2^{3k-1}$.
The adjacencies within each $V_i$ 
correspond to Hamming distances $2k-1$ and $2k+1,...,3k-1$.
Also, two vertices from distinct $V_i$'s are adjacent iff
their Hamming distance is between $2k$ and $3k-1$.
This leads to the following definition of $C_{3k-1}(m)$.
\begin{align*}
&V(C_{3k-1}(m)) = \{(v,i) \;|\; v \in \Z_2^{3k-1}, 1 \leq i \leq m\} \\
&(v,i) \sim (u,i) \;\; \text{ if } d(v,u) \in \{2k-1,2k+1,2k+2,...,3k-1\} \\
&(v,i) \sim (u,j) \;\; \text{ if } i \neq j \text{ and } d(v,u) \in \{2k,2k+1,2k+2,...,3k-1\}
\end{align*}
By the above discussion $C_{3k-1}(4) = C_{3k+1}$, and $C_{3k-1}(m)$ is 3ECTF for $m \geq 4$,
since any three vertices belong to an isomorphic copy of $C_{3k+1}$.
For the same reason $\mu_2(C_{3k-1}(m)) \geq \binom{2k}{k}$,
and for constant $k$ and large $m$ the graph is $\Theta(n)$-regular.

To introduce the second variation, consider the following graph.
$$ C'_{4k} = \left\langle \Z_2^{4k}, \{1,3,5,...,2k-1,4k\} \right\rangle \;. $$
By considering the two matchings on odd-parity and on even-parity vectors
we see that this is an Albert graph. In fact, each odd distance $d$ 
can be separately replaced by $4k-d$ to
yield another Albert graph.

$C_{3k+1}$ and $C'_{4k}$ are the first and last members of a simple sequence of 3ECTF graphs,  
$C_{k,j}$ for $j \in \{1,...,k\}$. These graphs are also defined in terms of the Hamming metric on the hypercube. 
We describe these graphs without proving their properties, 
since this is not needed henceforth.
$$ C_{k,j} = \left\langle \Z_2^{3k+j}, \{2k+1,2k+3,...,2k+(2j-1)\} \cup \{2(k+j),2(k+j)+1,...,3k+j\} \right\rangle .$$
%\begin{align*}
%&\left\langle \Z_2^{16}, \{11,12,13,14,15,16\} \right\rangle \\
%&\left\langle \Z_2^{17}, \{11,13,14,15,16,17\} \right\rangle \\
%&\left\langle \Z_2^{18}, \{11,13,15,16,17,18\} \right\rangle \\
%&\left\langle \Z_2^{19}, \{11,13,15,17,18,19\} \right\rangle \\
%&\left\langle \Z_2^{20}, \{11,13,15,17,19,20\} \right\rangle
%\end{align*}

\section{Twisted Graphs}\label{twistsect}

Having seen 3ECTF graphs with unbounded $\mu_2$, 
we move to the next construction in search of many such graphs.

We start with variation of a 3ECTF construction from Section 13.2 of \cite{cherlin2011}.
Given positive integers $m_0$, $m_1$, $m_2$ and $m_3$, we define
the twisted graph $G(m_0,m_1,m_2,m_3)$ as follows.
\begin{align*}
&V(G(m_0,m_1,m_2,m_3)) = \{(i,j,x) \;|\; i \in \{0,1,2,3\}, 1 \leq j \leq m_i, x \in \Z_4 \} \\
&(i,j,x) \sim (i,j,x+1) \;\; \text{ for all } i,j,x \\
&(i,j,x) \sim (i,j',x+2) \;\; \text{ for all } i,x,j \neq j' \\
&(i,j,x) \sim (i',j',x+3) \;\; \text{ for all } x,j,j',
(i,i') \in \{(0,1),(0,2),(0,3),(1,2),(2,3),(3,1)\} \;.
\end{align*}
\emph{Remark:} the graph $G(1,m_1,m_2,m_3)$ differs from $G(m_1,m_2,m_3)$ 
in Example 13.3 of~\cite{cherlin2011}, unless we switch two edges inside $V_0$. 
We do not know how to place both graphs on a common ground.

The following proposition reveals some of the structure of these graphs. 
Although it is a special case of Proposition~\ref{gtmk}, we believe that it is easier to follow 
and would make the more complicated Proposition~\ref{gtmk} more transparent.
\begin{prop}\label{gmmmm}
For $m_0,m_1,m_2,m_3 \geq 2$, the graph $G(m_0,m_1,m_2,m_3)$ is 3ECTF.
\end{prop}
\begin{proof}
For $i \in \{0,1,2,3\}$ let
$$V_i = \{(i,j,x) \in V(G) \;|\; x \in \Z_4, 1 \leq j \leq m_i \} .$$
Note that the restriction of $G$ to each $V_i$ is isomorphic to $A(m_i)$ (see Section~\ref{albertsect}).
Moreover, the subgraph induced on $V_i \cup V_{i'}$ is $A(m_i+m_{i'})$.
To see this, ''twist'' the $m_{i'}$ last $4$-cycles in $A(m_i+m_{i'})$, 
by sending $x \mapsto x+1 \bmod 4$. This affects only cross edges with $x+2 \to x+3$.
Therefore, we only need to consider triplets from distinct $V_i$'s. 

We first verify triangle-freeness. Along an edge of such a triangle, 
the $\Z_4$-coordinate changes from $x$ to $x \pm 3$, 
but three such numbers do not sum to zero $\bmod\;4$.

We next find a common neighbor for three independent vertices 
$(i,j,x)$, $(i',j',x')$, and $(i'',j'',x'')$ for distinct $i$, $i'$, and $i''$.
Suppose first that the parity of $x$ differs from those of $x'$ and $x''$. 
Which is the neighbor of $(i',j',x')$ in the $4$-cycle $(i,j,\cdot)$?
From parity considerations, and since $(i,j,x)\not\sim(i',j',x')$, it must be $(i,j,x+2)$. 
For similar reasons $(i,j,x+2)$ is the neighbor of $(i'',j'',x'')$ in that $4$-cycle. 
This implies that $(i,j''',x+2)$ is a common neighbor of 
$(i,j,x)$, $(i',j',x')$, and $(i'',j'',x'')$ for any $j''' \neq j$.

Now if $x$, $x'$ and $x''$ have equal parities,
then each of $(i',j',x')$ and $(i'',j'',x'')$ is adjacent to 
one of the elements of the set $\{(i,j,x-1),(i,j,x+1)\}$. 
If both are adjacent to the same one, it is the common neighbor and we are done. 
A similar solution may exist in different roles of the three vertices.
If all fails, then either $(i,i',i'') = (1,2,3)$ and $x=x'=x''$ or 
$(i,i',i'') \in \{(0,1,2),(0,2,3),(0,3,1)\}$ and $x=x''=x'+2$.
In this case, the common neighbor is in the fourth part $V_{i'''}$,
being $(i''',j''',x+1)$ and $(i''',j''',x+3)$, respectively, for any $j'''$.
\end{proof}

A \emph{tournament} is an orientation of a complete graph. 
Edges $(i,i') \in E(T)$ are denoted $i \rightarrow i'$.
Note that the set of ordered pairs $\{(0,1),(0,2),(0,3),(1,2),(2,3),(3,1)\}$,
that appears in the above definition of $G(m_0,m_1,m_2,m_3)$,
can be encoded by the following tournament that we call $T_4$.
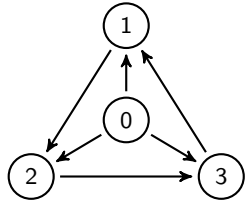
\begin{center}
\begin{tikzpicture}
[
main node/.style={circle,draw,thick,font=\sffamily\small},
main edge/.style={->,shorten <=2pt,shorten >=2pt,>=stealth',node distance=1cm,thick},
]
\node[main node] at (0,2) (1) {1};
\node[main node] at (0,0.75) (0) {0};
\node[main node] at (-1.25,0) (2) {2};
\node[main node] at (1.25,0) (3) {3};
\path[main edge]
    (0) edge node {} (1)
    (0) edge node {} (2)
    (0) edge node {} (3)
    (1) edge node {} (2)
    (2) edge node {} (3)
    (3) edge node {} (1);
\end{tikzpicture}
\end{center}
The tournament $T_4'$ is obtained from $T_4$ by reversing all edges. 
These two four-vertex tournaments are characterized by the property 
that every two vertices are connected by exactly one path of length two. 
In other words, 
for each pair of vertices one remaining vertex is attached to them in the same way,
while the other one is attached in opposite ways.
It is exactly this property that allowed us 
to find the common neighbor in the previous argument. 
A tournament is \emph{shattered} if every three of its vertices
extend to a copy of either $T_4$ or $T_4'$.

We next associate a graph $G_T(m)$ to a shattered tournament $T$ and a positive integer $m$.
\begin{align*}
&V(G_T(m)) = \{(i,j,x) \;|\; i \in V(T), 1 \leq j \leq m, x \in \Z_4 \} \\
&(i,j,x) \sim (i,j,x+1) \;\; \text{ for all } i,j,x \\
&(i,j,x) \sim (i,j',x+2) \;\; \text{ for all } i,x,j \neq j' \\
&(i,j,x) \sim (i',j',x+3) \;\; \text{ for all } x,j,j',i \rightarrow i' \;.
\end{align*}
It is easy to apply the proof of Proposition~\ref{gmmmm} to general shattered tournaments 
and conclude the following broader statement.
\begin{cor}
The graph $G_T(m)$ is 3ECTF for every shattered tournament $T$ and every $m \geq 2$.
\end{cor}
The graphs that were constructed so far in this section have $\mu_2(G)=2$ and $\mu_3(G)=1$.
The next example integrates them with $C_{3k-1}(m)$, thus increasing the $2$-multiplicity.
To this end, we first define a similar {\em twist} function in the hypercube. 
For $x \in \Z_2^n$, define
$$ \tau((x_1,x_2,x_3,x_4,...,x_n)) = (x_2,x_1+1,x_3,x_4,...,x_n) .$$
Note that $\tau^4(x)=x$ for all $x \in \Z_2^n$.

We turn to extend this construction and associate a graph $G_T(m,k)$ 
with every shattered tournament $T$ and positive numbers $m, k$.
\begin{align*}
&V(G_T(m,k)) = \{(i,j,x) \;|\; i \in V(T), 1 \leq j \leq m, x \in \Z_2^{3k-1} \} \\
&(i,j,x) \sim (i,j,x') \;\;  
\text{ if } d(x,x') \in \{2k-1,2k+1,2k+2,...,3k-1\} \\
&(i,j,x) \sim (i,j',x') \;\; 
\text{ if } j \neq j' \text{ and } d(x,x') \in \{2k,2k+1,...,3k-1\} \\
&(i,j,x) \sim (i',j',x') \;\; \text{ if } i \rightarrow i' 
\text{ and } d(x,\tau(x')) \in \{2k,2k+1,...,3k-1\} \;.
\end{align*}
Note that $G_T(m)$ is isomorphic to $G_T(m,1)$.
\begin{prop}\label{gtmk}
For $m \geq 2$, $k \geq 1$, and a shattered tournament $T$, the graph $G_T(m,k)$ is 3ECTF.
\end{prop}
\begin{proof}
We adapt the proof of Proposition~\ref{gmmmm} to $k > 1$. 
For $i \in V(T)$, we partition the vertices as follows:
$$V_i = \{(i,j,x) \;|\; x \in \Z_2^{3k-1}, 1 \leq j \leq m \} .$$
The induced graph on each $V_i$ is isomorphic to $C_{3k-1}(m)$.
Since $\tau$ is an isometry, also $V_i \cup V_{i'}$ induces a copy of $C_{3k-1}(2m)$.
By our previous comments these subgraphs are 3ECTF.

We next show that three vertices from distinct parts never form a triangle.
For each edge between parts $(i,j,x) \sim (i',j',x')$ we have $d(x,\tau(x')) \geq 2k$.
Therefore at least $2k-2$ of the last $3k-3$ coordinates of $x$ and $x'$ differ.
Suppose, toward a contradiction, that a triangle of such three edges exists, 
and consider two cases.
\begin{enumerate}
\item 
Suppose that along some of the edges the Hamming distance of the last $3k-3$ coordinates 
is at least $2k-1$. Then the three differences cannot add up to $0$, 
because at least one of the last $3k-3$ coordinates flips three times.
\item 
If, on the other hand, they all equal $2k-2$, then in each such edge $d(x, \tau(x'))=2k$, 
and hence $\text{parity}(x)=\text{parity}(\tau(x'))$. 
But $\tau$ switches parity, so the parity changes three times around the triangle 
-- a contradiction.
\end{enumerate}

Next, we seek a common neighbor for the independent set 
$\{(i_1,j_1,x_1), (i_2,j_2,x_2), (i_3,j_3,x_3)\}$, with $i_1, i_2, i_3$ distinct.
Define $\tau_{12} = \tau$ if $i_1 \rightarrow i_2$, and $\tau_{12} = \tau^{-1}$ otherwise.
Correspondingly define $\tau_{23}$ and so on.
Also, define the following two sets of three distances.
\begin{align*}
d_{12} = d(x_1,\tau_{12}(x_2)) && d^*_{12} = d(\tau_{31}(x_1),\tau_{32}(x_2)) \\
d_{23} = d(x_2,\tau_{23}(x_3)) && d^*_{23} = d(\tau_{12}(x_2),\tau_{13}(x_3)) \\
d_{31} = d(x_3,\tau_{31}(x_1)) && d^*_{31} = d(\tau_{23}(x_3),\tau_{21}(x_1))
\end{align*}
By the independence assumption, $d_{12}, d_{23}, d_{31}\le 2k-1$.
Note that $d_{12}$ and $d^*_{12}$ are defined by the same two vectors
up to three applications of $\tau$ or $\tau^{-1}$, 
so that $d^*_{12} = d_{12} \pm 1$,  $d^*_{23} = d_{23} \pm 1$, and $d^*_{31} = d_{31} \pm 1$.

Let us assume first $d^*_{12} \leq 2k-2$.
Apply Lemma~\ref{tri} to the vectors $x_3$, $\tau_{31}(x_1)$, and $\tau_{32}(x_2)$, 
with $a = k$ and $b = c = k-1$ and obtain some $v \in \Z_2^{3k-1}$ with
\begin{align*} 
d(v,x_3) \leq k, && d(v,\tau_{31}(x_1)) \leq k-1, && d(v,\tau_{32}(x_2)) \leq k-1.
\end{align*}
If $d(v,x_3) = k$ then $(i_3,j_3,\bar{v})$ is a common neighbor, 
while for $d(v,x_3) = k-1$ it is $(i_3,j,\bar{v})$ for any $j \neq j_3$.
For $d(v,x_3) \leq k-2$ any $j$ would do. 

By applying the same argument to $d^*_{23}$ and $d^*_{31}$, 
we may assume that $d^*_{12}, d^*_{23}, d^*_{31}\ge 2k-1$.
In particular, they exceed $d_{12}$, $d_{23}$ and $d_{31}$ by one, respectively.

The \emph{restricted parity} of $x \in \Z_2^n$, 
is defined as the parity of its first two coordinates. 
We claim that, under the above assumption, 
$x_1$, $x_2$ and $x_3$ all have the same restricted parity.
Otherwise, say $x_3$ is the exception. This implies several further properties.
\begin{enumerate}
\item 
The vectors $\tau_{21}(x_1)$ and $\tau_{23}(x_3)$ have different restricted parity.
\item 
Since $d^*_{12} = d_{12} + 1$,
at least one of the first two coordinates of
$\tau_{31}(x_1)$ and $\tau_{32}(x_2)$ that appear in the definition of $d^*_{12}$ must differ.
But then, from restricted parity considerations, they both differ.
\item 
Also by restricted parity and by the previous property, 
$x_3$ agrees either with $\tau_{31}(x_1)$ or with $\tau_{32}(x_2)$
on the first two coordinates, and disagrees on them with the other one.
Say it disagrees with $\tau_{31}(x_1)$.
\end{enumerate}
Properties 1 and 3 imply that $d^*_{31} < d_{31}$, an already settled case.
Therefore, the three vectors must have the same restricted parity.
Now, by the reasoning of property 2, each of the three following vector pairs
differs in both of the first two coordinates.
\begin{align*}
\tau_{31}(x_1),\;\tau_{32}(x_2) && 
\tau_{12}(x_2),\;\tau_{13}(x_3) &&
\tau_{23}(x_3),\;\tau_{21}(x_1)
\end{align*}

Here the shattered tournament comes in. There is
an $i_4$ for which $\{i_1,i_2,i_3,i_4\}$ 
induces either a $T_4$ or a $T_4'$ tournament.
Therefore exactly one of $(\tau_{41},\tau_{42})$ 
differs from its counterpart in $(\tau_{31},\tau_{32})$.
Consequently, $\tau_{41}(x_1)$ and $\tau_{42}(x_2)$ agree on their first two coordinates.
Denoting their distance by $d^+_{12}$, we have
$$ d^+_{12} = d^*_{12} - 2 = d_{12} - 1 \leq 2k-2 ,$$
and likewise $d^+_{23}, d^+_{31}\le 2k-2$.
Apply Lemma~\ref{tri} to 
$\tau_{41}(x_1)$, $\tau_{42}(x_2)$, and $\tau_{43}(x_3)$ with $a=b=c=k-1$, 
to obtain $v \in \Z_2^{3k-1}$ at distance at most $k-1$ from each of the three.
This yields the desired common neighbor $(i_4,j,\bar{v})$ for any possible $j$.
\end{proof}

As every two vertices of $G_T(2,k)$ are covered by some embedded copy of $C_{3k+1}$, 
we have $\mu_2(G_T(2,k)) \geq \binom{2k}{k}$ by Proposition~\ref{hc2}.
One can verify that this is in fact an equality.
Since there are $2^{\Omega(n^2/2^{6k})}$ nonisomorphic tournaments on $n/2^{3k}$ vertices,
and for large $n$ almost all of them are shattered, Proposition~\ref{gtmk} yields

\begin{cor}\label{many2}
For every $\mu \in \N$, there are {\Large $$ 2^{\Omega\left(\frac{n^2}{(\mu\log\mu)^3}\right)}$$}3ECTF 
graphs with up to $n$ vertices,
in which every pair of independent vertices has at least $\mu$ common neighbors.
\end{cor}

\bibliographystyle{abbrv}
{ \small
\bibliography{existentially-complete}
}

\end{document}